\documentclass[12 pt]{amsart}
\usepackage{amssymb,latexsym,amsmath,amscd,amsthm,graphicx, color}
\usepackage[all]{xy}
\usepackage{pgf,tikz}
\usepackage{mathrsfs}
\usepackage{cite}

\usetikzlibrary{arrows}
\definecolor{uuuuuu}{rgb}{0.26666666666666666,0.26666666666666666,0.26666666666666666}
\definecolor{xdxdff}{rgb}{0.49019607843137253,0.49019607843137253,1.}
\definecolor{ffqqqq}{rgb}{1.,0.,0.}

\definecolor{uuuuuu}{rgb}{0.26666666666666666,0.26666666666666666,0.26666666666666666}
\definecolor{qqwuqq}{rgb}{0.,0.39215686274509803,0.}
\definecolor{zzttqq}{rgb}{0.6,0.2,0.}
\definecolor{xdxdff}{rgb}{0.49019607843137253,0.49019607843137253,1.}
\definecolor{qqqqff}{rgb}{0.,0.,1.}
\definecolor{cqcqcq}{rgb}{0.7529411764705882,0.7529411764705882,0.7529411764705882}

\definecolor{uuuuuu}{rgb}{0.26666666666666666,0.26666666666666666,0.26666666666666666}
\definecolor{qqwuqq}{rgb}{0.,0.39215686274509803,0.}
\definecolor{zzttqq}{rgb}{0.6,0.2,0.}
\definecolor{xdxdff}{rgb}{0.49019607843137253,0.49019607843137253,1.}
\definecolor{qqqqff}{rgb}{0.,0.,1.}
\definecolor{cqcqcq}{rgb}{0.7529411764705882,0.7529411764705882,0.7529411764705882}

\theoremstyle{plain}

\newtheorem{theorem}[subsection]{Theorem}

\newtheorem{lemma}[subsection]{Lemma}

\newtheorem{prop}[subsection]{Proposition}
\newtheorem{defi}[subsection]{Definition}

\theoremstyle{definition}

\newtheorem{cor}[subsection]{Corollary}

\newtheorem{remark}[subsection]{Remark}

\newtheorem{note}[subsection]{Note}

\newcommand{\uu}{\cup}
\newcommand{\ii}{\cap}
\newcommand{\UU}{\bigcup}

\newcommand{\sci}{\subset}
\newcommand{\es}{\emptyset}
\newcommand{\set}[1]{\{#1\}}

\newcommand{\ga}{\alpha}
\newcommand{\gb}{\beta}

\newcommand{\gd}{\delta}
\renewcommand{\gg}{\gamma}

\newcommand{\go}{\omega}

\newcommand{\gs}{\sigma}
\newcommand{\gt}{\tau}

\newcommand{\tit}{\textit}

\newcommand{\D}[1]{\mathbb{#1}}

\newcommand{\te}{\text}

\begin{document}
To appear, Discrete and Continuous Dynamical Systems - Series A
\title{Least upper bound of the exact  formula for optimal quantization of some uniform Cantor distributions}

\author{ Mrinal Kanti Roychowdhury}
\address{School of Mathematical and Statistical Sciences\\
University of Texas Rio Grande Valley\\
1201 West University Drive\\
Edinburg, TX 78539-2999, USA.}
\email{mrinal.roychowdhury@utrgv.edu}

\thanks{The research of the author was supported by U.S. National Security Agency (NSA) Grant H98230-14-1-0320}

\date{}
\maketitle

\pagestyle{myheadings}\markboth{Mrinal Kanti Roychowdhury}{Least upper bound of the exact  formula for optimal quantization of some uniform Cantor distributions}

\begin{abstract} The quantization scheme in probability theory deals with finding a best approximation of a given probability distribution by a probability distribution that is supported on finitely many points. Let $P$ be a Borel probability measure on $\mathbb R$ such that $P=\frac 12 P\circ S_1^{-1}+\frac 12 P\circ S_2^{-1},$ where $S_1$ and $S_2$ are two contractive similarity mappings given by $S_1(x)=rx$ and $S_2(x)=rx+1-r$ for $0<r<\frac 12$ and $x\in \mathbb R$. Then, $P$ is supported on the Cantor set generated by $S_1$ and $S_2$. The case $r=\frac 13$ was treated by Graf and Luschgy who gave an exact formula for the unique optimal quantization of the Cantor distribution $P$ (Math. Nachr., 183 (1997), 113-133). In this paper, we compute the precise range of $r$-values to which Graf-Luschgy formula extends.
\end{abstract}

\section{Introduction}

The most common form of quantization is rounding-off. Its purpose is to reduce the cardinality of the representation space, in particular, when the input data is real-valued. It has broad applications in communications, information theory, signal processing and data compression (see \cite{GG, GL1, GL2, GN, P, Z1, Z2}).  Let $\D R^d$ denote the $d$-dimensional Euclidean space equipped with the Euclidean norm $\|\cdot\|$, and let $P$ be a Borel probability measure on $\D R^d$. Then, the $n$th \textit{quantization
error} for $P$, with respect to the squared Euclidean distance, is defined by
\begin{equation*} \label{eq1} V_n:=V_n(P)=\inf \Big\{V(P, \ga) : \ga \subset \mathbb R^d, \text{ card}(\ga) \leq n \Big\},\end{equation*}
where $V(P, \ga)= \int \min_{a\in\alpha} \|x-a\|^2 dP(x)$ represents the distortion error due to the set $\ga$ with respect to the probability distribution $P$.
A set $\ga\sci \D R^d$ is called an optimal set of $n$-means for $P$ if $V_n(P)=V(P, \ga)$.
For a finite set $\ga \sci \D R^d$ and $a\in \ga$, by $M(a|\ga)$ we denote the set of all elements in $\D R^d$ which are nearest to $a$ among all the elements in $\ga$, i.e.,
\[M(a|\ga)=\set{x \in \D R^d : \|x-a\|=\min_{b \in \ga}\|x-b\|}.\]
$M(a|\ga)$ is called the \tit{Voronoi region} generated by $a\in \ga$. On the other hand, the set $\set{M(a|\ga) : a \in \ga}$ is called the \tit{Voronoi diagram} or \tit{Voronoi tessellation} of $\D R^d$ with respect to the set $\ga$.
\begin{defi} \label{defi00}
A set $\ga\sci \D R^d$ is called a \tit{centroidal Voronoi tessellation} (CVT) with respect to a probability distribution $P$ on $\D R^d$,  if it satisfies the following two conditions:

$(i)$ $P(M(a|\ga)\ii M(b|\ga))=0$ for $a, b\in \ga$, and $a \neq b$;

$(ii)$ $E(X : X \in M(a|\ga))=a$ for all $a\in \ga$,

where $X$ is a random variable with distribution $P$, and $E(X : X \in M(a|\ga))$ represents the conditional expectation of the random variable $X$ given that $X$ takes values in $M(a|\ga)$.
\end{defi}
A Borel measurable partition $\set{A_a : a\in \ga}$ is called a \tit{Voronoi partition} of $\D R^d$ with respect to the probability distribution $P$, if $P$-almost surely  $A_a\sci M(a|\ga)$ for all $a\in \ga$.
Let us now state the following proposition (see \cite{GG, GL2}).
\begin{prop} \label{prop0}
Let $\ga$ be an optimal set of $n$-means with respect to a probability distribution $P$, $a \in \ga$, and $M (a|\ga)$ be the Voronoi region generated by $a\in \ga$.
Then, for every $a \in\ga$,

$(i)$ $P(M(a|\ga))>0$, $(ii)$ $ P(\partial M(a|\ga))=0$, $(iii)$ $a=E(X : X \in M(a|\ga))$, and $(iv)$ $P$-almost surely the set $\set{M(a|\ga) : a \in \ga}$ forms a Voronoi partition of $\D R^d$.
\end{prop}
If $\ga$ is an optimal set of $n$-means and $a\in \ga$, then by Proposition~\ref{prop0}, we see that $a$ is the centroid of the Voronoi region $M(a|\ga)$ associated with the probability measure $P$, i.e., for a Borel probability measure $P$ on $\D R^d$, an optimal set of $n$-means forms a CVT of $\D R^d$; however, the converse is not true in general (see \cite{DFG, DR, R1, R2}).

Let $S_1$ and $S_2$ be two contractive similarity mappings on $\D R$ such that $S_1(x)= rx$ and $S_2(x)=rx +1-r$ for $0<r<\frac 12$, and $P=\frac 12 P\circ S_1^{-1}+\frac 12 P\circ S_2^{-1}$, where $P\circ S_i^{-1}$ denotes the image measure of $P$ with respect to
$S_i$ for $i=1, 2$ (see \cite{H}). Then, $P$ is a unique Borel probability measure on $\D R$ which has support the limit set generated by $S_1$ and $S_2$. By a word $\gs$ of length $k$, where $k\geq 1$, over the alphabet $\set{1, 2}$, it is meant that $\gs:=\gs_1\gs_2\cdots \gs_k\in \set{1, 2}^k$, and write $S_\gs:=S_{\gs_1}\circ S_{\gs_2}\circ \cdots \circ S_{\gs_k}$. For $\gs:=\gs_1\gs_2\cdots \gs_k\in \set{1, 2}^k$ and
$\tau:=\tau_1\tau_2\cdots \tau_\ell$ in $\{1, 2\}^\ell$, $k, \ell\geq 1$, by
$\gs\tau:=\gs_1\cdots \gs_k\tau_1\cdots \tau_\ell$ we mean the word obtained from the
concatenation of the words $\gs$ and $\tau$. A word of length zero is called the empty word and is denoted by $\es$.  For the empty word $\es$, by $S_\es$ we mean the identity mapping on $\D R$, and write $J:=J_\es=S_\es([0,1])=[0, 1]$. Then, the set $C:=\bigcap_{k\in \mathbb N} \bigcup_{\gs \in \{1, 2\}^k} J_\gs$ is known as the \textit{Cantor set} generated by the two mappings $S_1$ and $S_2$, and equals the support of the probability measure $P$ given by $P=\frac 12 P\circ S_1^{-1}+\frac 12 P\circ S_2^{-1}$, where $J_\gs:=S_\gs(J)$. For any $\gs \in \set{1, 2}^k$, $k\geq 1$, the intervals $J_{\gs1}$ and $J_{\gs2}$ into which $J_\gs$ is split up at the $(k+1)$th level are called the basic intervals of $J_\gs$.
\begin{defi}  \label{def1} For $n\in \D N$ with $n\geq 2$ let $\ell(n)$ be the unique natural number with $2^{\ell(n)} \leq n<2^{\ell(n)+1}$. For $I\sci \set{1, 2}^{\ell(n)}$ with card$(I)=n-2^{\ell(n)}$ let $\gb_n(I)$ be the set consisting of all midpoints $a_\gs$ of intervals $J_\gs$ with $\gs \in \set{1,2}^{\ell(n)} \setminus I$ and all midpoints $a_{\gs 1}$, $a_{\gs 2}$ of the basic intervals of $J_\gs$ with $\gs \in I$. Formally, $\gb_n(I)=\set{a_\gs : \gs \in \set{1,2}^{\ell(n)} \setminus I} \uu \set{a_{\gs 1} : \gs \in I} \uu \set {a_{\gs 2} : \gs \in I}.$
Moreover,
\begin{equation*}\label{eq11}
\int\mathop{\min}\limits_{a\in\gb_n(I)} \|x-a\|^2 dP=\frac{1}{ 2^{\ell(n)}} r^{2\ell(n)} V \Big(2^{\ell(n)+1}-n+r^2(n-2^{\ell(n)})\Big),
\end{equation*}
where $V$ is the variance.
\end{defi}

\begin{remark}
In the sequel, there are some ten digit decimal numbers. They are all rational approximations of some real numbers.
\end{remark}
In \cite{GL3}, Graf and Luschgy showed that $\gb_n(I)$ forms an optimal set of $n$-means for the probability distribution $P$ when $r=\frac 13$, and the $n$th quantization error is given by
\[V (\gb_n(I))=\frac{1}{18^{\ell(n)}}\cdot \frac 18 \Big(2^{\ell(n)+1}-n+\frac 19(n-2^{\ell(n)})\Big).\]
Notice that $\gb_n(I)$ forms a CVT of the Cantor set generated by the two mappings $S_1(x)=rx$ and $S_2(x)=rx+(1-r)$ for $0<r\leq \frac {5-\sqrt {17}}{2}$, i.e., if $0<r\leq 0.4384471872$ (up to ten significant digits). In \cite{R2}, we have shown that if $0.4371985206<r\leq 0.4384471872$ and $n$ is not of the form $2^{\ell(n)}$ for any positive integer $\ell(n)$, then there exists a CVT for the Cantor set for which the distortion error is smaller than the CVT given by $\gb_n(I)$ implying the fact that $\gb_n(I)$ does not form an optimal set of $n$-means for all $0<r\leq \frac {5-\sqrt {17}}{2}$. It was still not known what is the least upper bound of $r$ for which $\gb_n(I)$ forms an optimal set of $n$-means for all $n\geq 2$. In the following theorem, which is the main theorem of the paper, we give the answer of it.

\begin{theorem} \label{maintheorem} Let $\gb_n(I)$ be the set defined by Definition~\ref{def1}. Let $r\in (0, \frac 12)$ be the unique real number such that
\[(r-1)(r^4+r^2)=\frac{r^7+r^6+4 r^5-2 r^4-2 r^3-8 r^2+9 r-3}{6}.\]
Then, $r_0=r\approx 0.4350411707$ (up to ten significant digits) gives the least upper bound of $r$ for which the set $\gb_n(I)$ forms an optimal set of $n$-means for the uniform Cantor distribution $P$.
\end{theorem}
In the sequel instead of writing $r\approx 0.435041170$, we will write $r=0.435041170$. The arrangement of the paper is as follows:
In Definition~\ref{def2}, we have constructed a set $\gg_n(I)$, and in Proposition~\ref{prop51}, we have shown that $\gg_n(I)$ forms a CVT for the Cantor distribution $P$ if $0.3613249509\leq r\leq 0.4376259168$ (written up to ten decimal places). In Theorem~\ref{Th1}, we have proved that the set $\gb_n(I)$ forms an optimal set of $n$-means for $r=0.4350411707$. In Proposition~\ref{prop4}, we have shown that
 $V(P, \gb_n(I))=V(P, \gg_n(I))$ if $r=0.4350411707$, and $V(P, \gb_n(I))>V(P, \gg_n(I))$  if $0.4350411707< r\leq 0.4376259168<\frac{5-\sqrt{17}}2$. In Definition~\ref{def3}, we have constructed a set $\gd_n(I)$. In Proposition~\ref{prop5}, we have shown that if $0.4371985206<r\leq \frac{5-\sqrt{17}}2$, then $V(P, \gb_n(I))>V(P, \gd_n(I))$. Hence, if $0.4350411707<r\leq\frac{5-\sqrt{17}}2$, then the set $\gb_n(I)$ forms a CVT but does not form an optimal set of $n$-means implying the fact that the least upper bound of $r$ for which $\gb_n(I)$ forms an optimal set of $n$-means is given by $r_0=r\approx 0.4350411707$ (up to ten significant digits) which is Theorem~\ref{maintheorem}. Notice that the optimal sets of $n$-means and the $n$th quantization errors are not known for all Cantor distributions $P$ given by $P:=\frac 12 P\circ S_1^{-1}+\frac 12 P\circ S_2^{-1}$, where $S_1(x)= rx$ and $S_2(x)=rx +1-r$ for $0<r<\frac 12$. Thus, it is worthwhile to investigate the least upper bound of $r$ for which the exact formula to determine the optimal quantization given by Graf-Luschgy works.

\section{Preliminaries}
As defined in the previous section, let $S_1$ and $S_2$ be the two similarity mappings on $\D R$ given by $S_1(x)=rx$ and $S_2(x)=rx+1-r$, where $0<r<\frac 12$, and $P=\frac 12 P\circ S_1^{-1}+\frac 12 P\circ S_2^{-1}$ be the probability distribution on $\D R$ supported on the Cantor set generated by $S_1$ and $S_2$. Write $p_1=p_2=\frac 12$, and $s_1=s_2=r$. By $I^\ast$ we denote the set of all words over the alphabet $I:=\set{1,2}$ including the empty word $\es$. For $\go \in I^\ast$, by $s_\go$ we represent the similarity ratio of the composition mapping $S_\go$. Notice that the identity mapping has the similarity ratio one. Thus, if $\go:=\go_1\go_2\cdots \go_k$, then we have $s_{\go}=r^k$. Let $X$ be a random variable with probability distribution $P$. By $E(X)$ and $V:=V(X)$ we mean the expectation and the variance of the random variable $X$.
For words $\gb, \gg, \cdots, \gd$ in $\set{1,2}^\ast$, by $a(\gb, \gg, \cdots, \gd)$ we mean the conditional expectation of the random variable $X$ given $J_\gb\uu J_\gg \uu\cdots \uu J_\gd,$ i.e.,
\begin{equation*} \label{eq45} a(\gb, \gg, \cdots, \gd)=E(X|X\in J_\gb \uu J_\gg \uu \cdots \uu J_\gd)=\frac{1}{P(J_\gb\uu \cdots \uu J_\gd)}\int_{J_\gb\uu \cdots \uu J_\gd} x dP(x).
\end{equation*}
We now give the following lemma.

\begin{lemma} \label{lemma0001}
Let $P$ be the Cantor distribution and $i=1, 2$. A set $\ga \sci \D R$ is a CVT for $P$ if and only if $S_i(\ga)$ is a CVT for the image measure $P\circ S_i^{-1}$.
\end{lemma}
\begin{proof} Notice that $a, b\in \ga$ if and only if $S_i(a), S_i(b) \in S_i(\ga)$. Moreover, \[(P\circ S_i^{-1})\Big (M(S_i(a)|S_i(\ga))\ii M(S_i(b)|S_i(\ga))\Big)=P\Big(M(a|\ga)\ii M(b|\ga)\Big).\]
The last equation is true, since for any $c\in \ga$,
\begin{align*} &S_i^{-1}(M(S_i(c)|S_i(\ga)))=\set{S_i^{-1}(x) \in \D R : \|x-S_i(c)\|=\min_{b\in S_i(\ga)}\|x-b\|}\\
&=\set{y \in \D R :  \|S_i(y)-S_i(c)\|=\min_{b\in\ga}\|S_i(y)-S_i(b)\|}=\set{y \in \D R :  \|y-c\|=\min_{b\in\ga}\|y-b\|}\\
&=M(c|\ga).
\end{align*}
Hence, by Definition~\ref{defi00}, the lemma follows.
\end{proof}
From Lemma~\ref{lemma0001} the following corollary follows.

\begin{cor}\label{cor020}
Let $i=1, 2$, and let $\gb$ form a CVT for the image measure $P\circ S_i^{-1}$. Then, $S_i^{-1}(\gb)$ forms a CVT for the probability measure $P$.
\end{cor}

The following two lemmas are well-known and easy to prove (see \cite{GL3, R2}).
\begin{lemma} \label{lemma1}
Let $f : \mathbb R \to \mathbb R^+$ be Borel measurable and $k\in \mathbb N$, and $P$ be the probability measure on $\D R$ given by $P=\frac 12 P\circ S_1^{-1}+\frac 12 P\circ S_2^{-1}$. Then
\[\int f(x) dP(x)=\sum_{\sigma \in \{1, 2\}^k} \frac 1 {2^k} \int f \circ S_\sigma(x) dP(x).\]
\end{lemma}

\begin{lemma} \label{lemma2} Let $X$ be a random variable with the probability distribution $P$. Then,
\[E(X)=\frac 12 \te{ and }  V:=V(X)=\frac{1-r}{4(1+r)}, \te{ and } \int (x-x_0)^2 dP(x) =V (X) +(x_0-\frac 12)^2,\]
where $x_0\in \D R$.
\end{lemma}

We now give the following corollary.
\begin{cor} \label{cor1}
Let $\gs \in \set{1, 2}^k$ for $k\geq 1$, and $x_0 \in \mathbb R$. Then,
\begin{equation*} \label{eq234} \int_{J_\gs} (x-x_0)^2 dP(x) =\frac 1{2^k} \Big(r^{2k} V  +(S_\gs(\frac 12)-x_0)^2\Big).\end{equation*}
\end{cor}

\begin{note}  \label{note1} Corollary~\ref{cor1} is useful to obtain the distortion error.
By Lemma~\ref{lemma2}, it follows that the optimal set of one-mean is the expected value and the corresponding quantization error is the variance $V$ of the random variable $X$.
For $\gs \in \{1, 2\}^k$, $k\geq 1$, since $a(\gs)=E(X : X \in J_\gs)$,  using Lemma~\ref{lemma1}, we have
\begin{align*}
&a(\gs)=\frac{1}{P(J_\gs)} \int_{J_\gs} x \,dP(x)=\int_{J_\gs} x\, d(P\circ S_\gs)^{-1}(x)=\int S_\gs(x)\, dP(x)=E(S_\gs(X)).
\end{align*}
Since $S_1$ and $S_2$ are similarity mappings, it is easy to see that $E(S_j(X))=S_j(E(X))$ for $j=1, 2$ and so by induction, $a(\gs)=E(S_\gs(X))=S_\gs(E(X))=S_\gs(\frac 12)$
 for $\gs\in \{1, 2\}^k$, $k\geq 1$.
\end{note}

\begin{defi}\label{def2} For $n\in \D N$ with $n\geq 2$ let $\ell(n)$ be the unique natural number with $2^{\ell(n)} \leq n<2^{\ell(n)+1}$. Write
$\gg_2:=\set{a(1), a(2)}$ and $\gg_3:=\set{a(11, 121), a(122, 211),a(212, 22)}$. For $n\geq 4$, define $\gg_n:=\gg_n(I)$ as follows: \[\gg_n(I)=\left\{\begin{array}{cc}
\mathop{\uu}\limits_{\go\in I} S_\go(\gg_3)\UU \mathop{\uu}\limits_{\go \in \set{1, 2}^{\ell(n)-1}\setminus I} S_\go (\gg_2) & \te{ if } 2^{\ell(n)}\leq n\leq 3\cdot 2^{\ell(n)-1},\\
 \mathop{\uu}\limits_{\go\in \set{1, 2}^{\ell(n)-1}\setminus I} S_\go(\gg_3) \UU \mathop{\uu}\limits_{\go \in I} S_\go(\gg_4) & \te{ if } 3\cdot 2^{\ell(n)-1}<n< 2^{\ell(n)+1},
\end{array}
\right.\] where $I \sci \set{1, 2}^{\ell(n)-1}$ with $\te{card}(I)=n-2^{\ell(n)}$ if $2^{\ell(n)}\leq n\leq 3\cdot 2^{\ell(n)-1}$; and
$\te{card}(I)=n-3 \cdot 2^{\ell(n)-1}$ if $3\cdot 2^{\ell(n)-1}<n< 2^{\ell(n)+1}$.
\end{defi}

\begin{prop} \label{prop51}
Let $\gg_n:=\gg_n(I)$ be the set defined by Definition~\ref{def2}. Then, $\gg_n(I)$ forms a CVT for the Cantor distribution $P$ if $0.3613249509\leq r\leq 0.4376259168$ (written up to ten decimal places).
\end{prop}
\begin{proof} $\gg_2$ forms a CVT for any $0<r<\frac 12$. Using the similar arguments as \cite[Lemma~3.13]{R2}, we can show that $\gg_3$ forms a CVT for $P$ if
\begin{equation} \label{eq55}
0.3613249509<r<0.4376259168.
\end{equation} We now prove the proposition for $n\geq 4$. Let $\ell(n)$ be the unique natural number with $2^{\ell(n)} \leq n<2^{\ell(n)+1}$. Thus, for $n\geq 4$, we have $\ell(n)\geq 2$. Notice that the two similarity mappings $S_1$ and $S_2$ are increasing mappings in the sense that $S_i(x)<S_i(y)$ for all $x, y\in \D R$ with $x<y$, where $i=1, 2$.  This induces an order relation $\prec$ on $I^\ast$ as follows: for $\go, \gt \in I^\ast$, we write $\go \prec \gt$ if $S_\go(x)<S_\gt(y)$ for $x, y\in \D R$ with $x<y$. Let
\begin{equation} \label{eq10001} \go^{(1)}\prec \go^{(2)}\prec \go^{(3)}\prec \cdots \prec \go^{(2^{\ell(n)-1})}\end{equation}
be the order of the $2^{\ell(n)-1}$ elements in the set $I^{\ell(n)-1}$. For $1\leq i<2^{\ell(n)-1}$ and $\go^{(i)} \in I^{2^{\ell(n)-1}}$, $a(\go^{(i)}2)$ and $a(\go^{(i+1)}1)$ are, respectively, the midpoints of the basic intervals of $J_{\go^{(i)}2}$ and $J_{\go^{(i+1)}1}$ yielding
\[S_{\go^{(i)}2}(1)<\frac 12(a(\go^{(i)}2)+a(\go^{(i+1)}1))<S_{\go^{(i+1)}1}(0),\]
and so, for $n=2^{\ell(n)}$, the set $\gg_n$ forms a CVT for $P$. Let us now assume that $2^{\ell(n)}<n<2^{\ell(n)+1}$. Then, the set $\gg_n$ will form a CVT for $P$ if for $1\leq i<2^{\ell(n)-1}$ we can show that the following inequalities are true:
\begin{align}
S_{\go^{(i)}22}(1)&<\frac 12(a(\go^{(i)} 212,  \go^{(i)}22)+a(\go^{(i+1)}11, \go^{(i+1)}121))<S_{\go^{(i+1)}11}(0) \label{eq78},\\
S_{\go^{(i)}22}&(1)<\frac 12(a(\go^{(i)}212, \go^{(i)}22)+a(\go^{(i+1)}1))<S_{\go^{(i+1)}1}(0)  \label{eq79},\\
S_{\go^{(i)}22}&(1)<\frac 12(a(\go^{(i)}212, \go^{(i)}22)+a(\go^{(i+1)}11))<S_{\go^{(i+1)}11}(0)  \label{eq80}.
\end{align}
We call $\gt$ a predecessor of a word $\go\in I^\ast$, if $\go=\gt\gd$ for some $\gd\in I^\ast$. If  $\go^{(i)}$ and $\go^{(i+1)}$ have a common predecessor $\gt$, then notice that the points $a(\go^{(i)} 212,  \go^{(i)}22)$ and $a(\go^{(i+1)}11, \go^{(i+1)}121)$ are reflections of each other about the point $S_\gt(\frac 12)$, i.e., $\frac 12(a(\go^{(i)} 212,  \go^{(i)}22)+a(\go^{(i+1)}11, \go^{(i+1)}121))=S_\gt(\frac 12)$, and $S_{\go^{(i)}22}(1)$ and $S_{\go^{(i+1)}1}(0)$ are in opposite sides of $S_\gt(\frac 12)$ , and so, the inequalities in \eqref{eq78} are true.
If  $\go^{(i)}$ and $\go^{(i+1)}$ have no common predecessor, i.e., the predecessor is the empty word $\es$, then the two points $a(\go^{(i)} 212,  \go^{(i)}22)$ and $a(\go^{(i+1)}11, \go^{(i+1)}121)$ are reflections of each other about the point $\frac 12$, i.e., $\frac 12(a(\go^{(i)} 212,  \go^{(i)}22)+a(\go^{(i+1)}11, \go^{(i+1)}121))=\frac 12$, and so, the inequalities in \eqref{eq78} are true. We now prove the inequalities in \eqref{eq79} and \eqref{eq80}.
To prove the inequalities in the following, by $\go^{(i)}\wedge\go^{(i+1)}$, we denote the common predecessor of the words $\go^{(i)}$ and $\go^{(i+1)}$. Notice that $s_{\go^{(i+1)}}=s_{\go^{(i)}}$ and $s_{\go^{(i+1)}1}=s_{\go^{(i)}}r $ for all $\go^{(i)}, \go^{(i+1)}\in I^{\ell(n)-1}$. We have
\begin{align} \label{eq87}
&a(\go^{(i)}212, \go^{(i)}22)+a(\go^{(i+1)}1)-2S_{\go^{(i)}22}(1)\\
&=\frac {1}{\frac 1{2^3}+\frac 1{2^2}} \Big(\frac 1 {2^3} S_{\go^{(i)}212}(\frac 12)+\frac 1 {2^2} S_{\go^{(i)}22} (\frac 12)\Big)+S_{\go^{(i+1)}1}(\frac 12)-2S_{\go^{(i)}22}(1) \notag\\
&=\frac {1}{3} \Big(S_{\go^{(i)}212}(\frac 12)-S_{\go^{(i)}22}(1)\Big)+\frac 23\Big( S_{\go^{(i)}22} (\frac 12)-S_{\go^{(i)}22}(1)\Big)+\Big(S_{\go^{(i+1)}1}(\frac 12)-S_{\go^{(i)}22}(1)\Big) \notag\\
&=s_{\go^{(i)}}\Big(\frac {1}{3}( S_{212}(\frac 12)-S_{22}(1))+\frac 23( S_{22} (\frac 12)-S_{22}(1))\Big)+\Big((1-2r)s_{\go^{(i)}\wedge\go^{(i+1)}} +\frac 12 s_{\go^{(i+1)}1}\Big) \notag\\
&=r^{\ell(n)-1}\Big(\frac {1}{3}( S_{212}(\frac 12)-S_{22}(1))+\frac 23( S_{22} (\frac 12)-S_{22}(1))\Big)+(1-2r)s_{\go^{(i)}\wedge\go^{(i+1)}}+\frac 12 r^{\ell(n)}.\notag
\end{align}
If $n=5$, i.e., when $\ell(n)=2$, then \eqref{eq87} reduces to
\begin{equation} \label{eq88} a(1212, 122)+a(21)-2S_{122}(1)=r \Big(\frac {1}{3}( S_{212}(\frac 12)-S_{22}(1))+\frac 23( S_{22} (\frac 12)-S_{22}(1))\Big)+(1-2r) +\frac 12 r^2.
\end{equation}
Let $\gt$ be the predecessor with the maximum length among all the predecessors of any two consecutive words $\go^{(i)}$ and $\go^{(i+1)}$ as defined in \eqref{eq10001}. Then, by \eqref{eq87} and \eqref{eq88},  we have
\begin{align*}
a(\gt 1212, \gt122)+a(\gt21)-2S_{\gt 122}(1)& \leq a(\go^{(i)}212, \go^{(i)}22)+a(\go^{(i+1)}1)-2S_{\go^{(i)}22}(1)\\
&\leq a(1212, 122)+a(21)-2S_{122}(1).
\end{align*}
Similarly, we can prove that
\begin{align*}
2 S_{\gt 21}(0)-a(\gt 1212, \gt122)-a(\gt21)& \leq 2 S_{\go^{(i+1)}1}(0)-a(\go^{(i)}212, \go^{(i)}22)-a(\go^{(i+1)}1)\\
&\leq 2 S_{21}(0)-a(1212, 122)- a(21).
\end{align*}
Thus, the inequalities in \eqref{eq79} will be true if we can prove that
\begin{equation} \label{eq91} S_{\gt 122}(1)<\frac 12(a(\gt 1 212, \gt 122)+a(\gt 2 1))<S_{\gt 2 1}(0). \end{equation}
Proceeding in the similar way, we can prove that the inequalities in \eqref{eq80} will be true if we can prove that
\begin{equation} \label{eq92}  S_{\gt 1 22}(1)<\frac 12(a(\gt 1 212, \gt 122)+a(\gt 211))<S_{\gt 211}(0). \end{equation}
Using Lemma~\ref{lemma0001}, we can say that the inequalities in \eqref{eq91} and \eqref{eq92} will be true if we can prove that
\begin{align}
S_{122}&(1)<\frac 12(a(1212, 122)+a(21))<S_{21}(0), \label{eq93}\te{ and } \\
S_{122}&(1)<\frac 12(a(1212, 122)+a(211))<S_{211}(0).  \label{eq94}
\end{align}
The inequalities in \eqref{eq93} are true if $0<r<0.4850084548$, and  the inequalities in \eqref{eq94} are true if $0 < r < 0.4847126592$. Combining these with \eqref{eq55}, we see that $\gg_n(I)$ forms a CVT for the Cantor distribution $P$ if $0.3613249509\leq r\leq 0.4376259168$ (written up to ten decimal places). Thus, the proof of the proposition is complete.
\end{proof}

\begin{prop} \label{prop34} For $n\geq 4$ let $\gg_n(I)$ be the set defined by Definition~\ref{def2}. Then,
\begin{align*}
&\int \mathop{\min}\limits_{a\in \gg_n(I)}(x-a)^2 dP\\
&=\left\{ \begin{array} {ll }
r^{2\ell(n)}V   \te{ if } n=2^{\ell(n)},\\
\frac 1{2^{\ell(n)-1}}r^{2(\ell(n)-1)} \Big(V_3\, (n-2^{\ell(n)}) +V_2 \,(3\cdot 2^{\ell(n)-1}-n)\Big)   \te{ if } 2^{\ell(n)}< n\leq 3 \cdot 2^{\ell(n)-1},\\
\frac 1{2^{\ell(n)-1}}r^{2(\ell(n)-1)} \Big(V_3\, (2^{\ell(n)+1}-n) +V_4\,(n-3\cdot 2^{\ell(n)-1})\Big) \te{ if } 3 \cdot 2^{\ell(n)-1}<n< 2^{\ell(n)+1},
\end{array}
\right.
\end{align*}
where $V_2:=V(P, \gg_2)$ and $V_3:=V(P, \gg_3)$, respectively, denote the distortion errors for the CVTs $\gg_2(I)$ and $\gg_3(I)$.
\end{prop}
\begin{proof}
  For $n=2^{\ell(n)}$, we have
\begin{align*}
\sum_{\go\in \set{1,2}^{\ell(n)}}\mathop{\int}\limits_{J_\go} (x-a(\go))^2 dP=\frac 1{2^{\ell(n)}}\sum_{\go\in \set{1,2}^{\ell(n)}} \int (x-a(\go))^2 d(P\circ S_\go^{-1})=r^{2\ell(n)}V.
\end{align*}
For $2^{\ell(n)}< n\leq 3 \cdot 2^{\ell(n)-1}$,
\begin{align*}
& \int \min_{a\in \gg_n(I)}(x-a)^2 dP=\sum_{\go \in I}\mathop{\int}\limits_{J_\go}  \min_{a\in S_\go(\gg_3)}(x-a)^2 dP+\sum_{\go \in \set{1, 2}^{\ell(n)-1}\setminus I}\mathop{\int}\limits_{J_\go} \min_{a\in S_\go(\gg_2)}(x-a)^2 dP\\
&=\sum_{\go \in I}\frac 1{2^{\ell(n)-1}} \int  \min_{a\in S_\go(\gg_3)}(x-a)^2 d(P\circ S_\go^{-1})\\
& \qquad \qquad +\sum_{\go \in \set{1, 2}^{\ell(n)-1}\setminus I}\frac 1{2^{\ell(n)-1}} \int  \min_{a\in S_\go(\gg_2)}(x-a)^2 d(P\circ S_\go^{-1})\\
&=\sum_{\go \in I}\frac 1{2^{\ell(n)-1}}r^{2(\ell(n)-1)}V_3+\sum_{\go \in \set{1, 2}^{\ell(n)-1}\setminus I}\frac 1{2^{\ell(n)-1}}r^{2(\ell(n)-1)}V_2\\
&=\frac 1{2^{\ell(n)-1}}r^{2(\ell(n)-1)} \Big(V_3\, \te{card}(I)+V_2 \,\te{card}(\set{1, 2}^{\ell(n)-1}\setminus I)\Big)\\
&=\frac 1{2^{\ell(n)-1}}r^{2(\ell(n)-1)} \Big(V_3\, (n-2^{\ell(n)}) +V_2 \,(3\cdot 2^{\ell(n)-1}-n)\Big).
\end{align*}

For $3 \cdot 2^{\ell(n)-1}<n< 2^{\ell(n)+1}$,
\begin{align*}
& \int \min_{a\in \gg_n(I)}(x-a)^2 dP=\sum_{\go \in \set{1, 2}^{\ell(n)-1}\setminus I}\mathop{\int}\limits_{J_\go} \min_{a\in S_\go(\gg_3)}(x-a)^2 dP+\sum_{\go \in I}\mathop{\int}\limits_{J_\go} \min_{a\in S_\go(\gg_4)}(x-a)^2 dP\\
&=\sum_{\go \in \set{1, 2}^{\ell(n)-1}\setminus I}\frac 1{2^{\ell(n)-1}} \mathop{\int} \min_{a\in S_\go(\gg_3)}(S_\go(x)-a)^2 dP+\sum_{\go \in I}\frac 1{2^{\ell(n)-1}}\mathop{\int} \min_{a\in S_\go(\gg_4)}(S_\go(x)-a)^2 dP\\
&=\frac 1{2^{\ell(n)-1}}r^{2(\ell(n)-1)} \Big(V_3\, \te{card}(\set{1, 2}^{\ell(n)-1}\setminus I)+V_4 \,\te{card}(I)\Big)\\
&=\frac 1{2^{\ell(n)-1}}r^{2(\ell(n)-1)} \Big(V_3\, (2^{\ell(n)+1}-n) +V_4\,(n-3\cdot 2^{\ell(n)-1})\Big).
\end{align*}
Thus, the proof of the proposition is complete.
\end{proof}

\section{Optimal sets of $n$-means for $r=0.4350411707$ and $n\geq 2$} \label{sec344}

Recall that $\gb_n(I)$ forms a CVT if $r=0.4350411707$. In this section, we state and prove the following theorem.
\begin{theorem} \label{Th1}
Let $n\geq 2$, and let $\gb_n(I)$ be the set given by Definition~\ref{def1}. Then, $\gb_n(I)$ forms an optimal set of $n$-means for $r=0.4350411707$.
\end{theorem}
To prove the theorem, we need some basic lemmas and propositions.

The following two lemmas are true. Due to technicality the proofs of them are not shown in the paper.
 \begin{lemma}\label{lemma11} Let $\ga:=\{a_1, a_2\}$ be an optimal set of two-means, $a_1<a_2$. Then, $a_1=a(1)=S_1(\frac 12) =0.2175$, $a_2=a(2)=S_2(\frac 12)=0.7825$, and the corresponding quantization error is $V_2=r^2 V=0.0186274$.
\end{lemma}

\begin{lemma} \label{lemma12}
The sets $\set{a(1), a(21), a(22)}$, and $\set{a(11), a(12), a(2)}$ form two optimal sets of three-means with quantization error $V_3=0.0110764$.
\end{lemma}
We now prove the following lemma.
\begin{lemma} \label{lemma13}
Let $\ga_n$ be an optimal set of $n$-means for $n\geq 2$. Then, $\ga_n\ii [0, r)\neq \es$ and $\ga_n\ii (1-r, 1]\neq \es$.
\end{lemma}

\begin{proof} For $n=2$ and $n=3$, the statement of the lemma follows from Lemma~\ref{lemma11} and Lemma~\ref{lemma12}.   Let us now prove that the lemma is true for $n\geq 4$. Consider the set of four points $\gb$ given by $\gb:=\set{a(\gs) : \gs \in \set{1, 2}^2}$. Then,
\[\int\min_{a \in \gb} (x-a)^2 dP=\sum_{\gs \in \set{1, 2}^2} \int_{J_{\gs}}(x-a(\gs))^2 dP=0.00352544.\]
Since $V_n$ is the $n$th quantization error for $n\geq 4$, we have $V_n\leq V_{4}\leq 0.00352544$.
Let $\ga_n$ be an optimal set of $n$-means. Write $\ga_n:=\set{a_1, a_2, \cdots, a_n}$, where $0< a_1<a_2<\cdots <a_n< 1$.
If $a_1\geq r$, using Corollary~\ref{cor1}, we have
\[V_n\geq \int_{J_1}(x-a_1)^2 dP\geq \int_{J_1}(x-r)^2 dP=\frac 12 \Big(r^2V+(a(1)-r)^2\Big)=0.0329713>V_{4}\geq V_n,\]
which is a contradiction. Thus, we can assume that $a_1<r$. Similarly, we can show that $a_n>(1-r)$. Thus, we see that if $\ga_n$ is an optimal set of $n$-means with $n\geq 2$, then
$\ga_n\ii [0, r)\neq \es$ and $\ga_n\ii (1-r, 1]\neq \es$. Thus, the lemma is yielded.
\end{proof}

The following lemma is a modified version of Lemma~4.5 in \cite{GL3}, and the proof follows similarly.
\begin{lemma} \label{lemma14}
Let $n\geq 2$, and let $\ga_n$ be an optimal set of $n$-means such that $\ga_n\ii J_1\neq \es$, $\ga_n\ii J_2\neq \es$, and $\ga_n\ii (r, 1-r)=\es$. Further assume that the Voronoi region of any point in $\ga_n\ii J_1$ does not contain any point from $J_2$, and the Voronoi region of any point in $\ga_n\ii J_2$ does not contain any point from $J_1$. Set $\ga_1:=\ga_n\ii J_1$ and $\ga_2:=\ga_n\ii J_2$, and $j:=\te{card}(\ga_1)$.
Then, $S_1^{-1}(\ga_1)$ is an optimal set of $j$-means and $S_2^{-1}(\ga_2)$ is an optimal set of $(n-j)$-means. Moreover,
\[V_n=\frac 12 r^2(V_j+V_{n-j}).\]
\end{lemma}

\begin{remark}
Lemma~4.5 in \cite{GL3} does not work for all $0<r<\frac 12$. Due to that we have added an extra condition to Lemma~4.5 in \cite{GL3} to work for all $0<r<\frac 12$.
\end{remark}

\begin{lemma} \label{lemma15}  Let $\ga_4$ be an optimal set of four-means. Then, $\ga_4:=\set{a(11), a(12), a(21), a(22)}$, and the quantization error is $V_4=0.00352544$.
\end{lemma}
\begin{proof}
Consider the four-point set $\gb$ given by
$\gb:=\set{a(\gs) : \gs\in\set{1, 2}^2}$.
Then,
\[\int\min_{a \in \gb} (x-a)^2 dP=\sum_{\gs \in \set{1, 2}^2} \mathop{\int}\limits_{J_{\gs}}(x-a(\gs))^2 dP=0.00352544.\]
Since $V_4$ is the $n$th quantization error for $n=4$, we have $V_4\leq 0.00352544$. Let $\ga_4:=\set{a_1, a_2, a_3, a_4}$, where $0< a_1<a_2<a_3<a_4< 1$, be an optimal set of four-means. If $a_1>0.20>0.189261=S_{11}(1)$, using Corollary~\ref{cor1}, we have
\[V_4\geq \int_{J_{11}}(x-a_1)^2 dP\geq \int_{J_{11}}(x-0.20)^2 dP=0.00365705>V_4,\]
which is a contradiction. So, we can assume that $a_1\leq 0.20$. Similarly, $a_4\geq 0.80$. We now show that $\ga_4$ does not contain any point from $(r, 1-r)$. Suppose that $\ga_4$ contains a point from $(r, 1-r)$. Then, due to Lemma~\ref{lemma13}, without any loss of generality, we can assume that $a_2\in (r, 1-r)$, and $1-r\leq a_3< a_4$. Two cases can arise:

Case 1: $a_2 \in [\frac 12, 1-r)$.

Then, $a_1\leq 0.20<S_{121}(0)<S_{121}(1)=0.328117<\frac 12(0.20 +\frac{1}{2})=0.35<0.352705=S_{122}(0)$. Notice that $a(11, 121)=0.158736<0.20$, and thus,
\[V_4\geq \int_{J_{11}\uu J_{121}}(x-a(11, 121))^2 dP=0.00404695>V_4,\]
which is a contradiction.

Case 2: $a_2\in (r, \frac 12]$.

Then, $S_{1211} (1)=0.2816<0.281742=\frac 12(a(11, 1211)+r)<0.292297=S_{1212}(0)$ implying the fact that $J_{11}\uu J_{1211}\sci M(a(11, 1211)|\ga_4)$ and $J_{122}\sci M(r|\ga_4)$. Again,
\begin{align} \label{eq89}
&\int_{J_2}\min_{a\in \set{a_2, a_3, a_4}}(x-a)^2 dP\geq \int_{J_2}\min_{a\in S_2(\ga_3)}(x-a)^2 dP=\frac 12  \int_{J_2}\min_{a\in S_2(\ga_3)}(x-a)^2 d(P\circ S_2^{-1})\\
&=\frac 12  \int \min_{a\in S_2(\ga_3)}(S_2(x)-a)^2 dP=\frac 12  \int \min_{a\in \ga_3}(S_2(x)-S_2(a))^2 dP=\frac 12 r^2 V_3,\notag
\end{align}
where $\ga_3$ is an optimal set of three-means as given by Lemma~\ref{lemma12}.
Thus, we obtain
\begin{align*}  V_4 \geq \int_{J_{11}\uu J_{1211}}(x-a(11, 1211))^2 dP+\int_{J_{122}}(x-r)^2 dP+\frac 12 r^2 V_3=0.00366173>V_4
\end{align*}
which gives a contradiction.

Hence, we can assume that $\ga_4$ does not contain any point from the open interval $(r, 1-r)$. We now show that $\te{card}(\ga_4\ii J_1)=\te{card}(\ga_4\ii J_2)=2$. For the sake of contradiction, assume that $a_1\in J_1$ and $\set{a_2, a_3, a_4}\sci J_2$. If the Voronoi region of $a_2$ does not contain any point from $J_1$, then
\[V_4\geq \int_{J_1}(x-a(1))^2 dP=0.00931372>V_4,\]
which leads to a contradiction. So, we can assume that the Voronoi region of $a_2$ contains points from $J_1$. Then, $\frac 12(a_1+a_2)<r$ implying $a_1<2r-a_2\leq 2r-(1-r)=3r-1=0.305124<0.312534=S_{12122}(0)$. Notice that $a(11, 1211, 12121)=0.144047<0.305124$, and $S_{12121}(1)=0.30788<0.354503=\frac 12(a(11, 1211, 12121)+(1-r))$. Then, using \eqref{eq89}, we have
\begin{align*}
V_4&\geq \int_{J_{11}\uu J_{1211}\uu J_{12121}}(x-a(11, 1211, 12121))^2 dP+\int_{J_{12122}\uu J_{122}}(x-0.305124)^2 dP+\frac 12 r^2 V_3\\
&=0.00528016>V_4,
\end{align*}
which is a contradiction. Thus, $\te{card}(\ga_4\ii J_1)=1$ and $\te{card}(\ga_4\ii J_2)=3$ give a contradiction. Since ($\te{card}(\ga_4\ii J_1)=3$ and $\te{card}(\ga_4\ii J_2)=1$) is a reflection of the case ($\te{card}(\ga_4\ii J_1)=1$ and $\te{card}(\ga_4\ii J_2)=3$) about the point $\frac 12$, we can say that $\te{card}(\ga_4\ii J_1)=3$ and $\te{card}(\ga_4\ii J_2)=1$ also yield a contradiction. Again, we have seen that $\ga_4\ii J_1\neq \es$ and $\ga_4\uu J_2\neq \es$. Thus, we have $\te{card}(\ga_4\ii J_1)=\te{card}(\ga_4\ii J_2)=2$. Since $P$ has symmetry about the point $\frac 12$, i.e., if two intervals of equal lengths are equidistant from the point $\frac 12$ then they have the same probability, and $\te{card}(\ga_4\ii J_1)=\te{card}(\ga_4\ii J_2)=2$, we can assume that the boundary of the voronoi regions of $a_2$ and $a_3$ passes through the point $\frac 12$, i.e., the Voronoi region of any point in $\ga_4\ii J_1$ does not contain any point from $J_2$, and the Voronoi region of any point in $\ga_4 \ii J_2$ does not contain any point from $J_1$. Hence,
 By Lemma~\ref{lemma14}, both $S_1^{-1}(\ga_4\ii J_1)$ and $S_2^{-1}(\ga_4\ii J_2)$ are optimal sets of two-means, i.e., $S_1^{-1}(\ga_4 \ii J_1)=S_2^{-1}(\ga_4 \ii J_2)=\set{a(1), a(2)}$ yielding $\ga_4 \ii J_1=\set{a(11), a(12)}$ and $\ga_4\ii J_2=\set{a(21), a(22)}$. Thus, we have $\ga_4=\set{a(11), a(12), a(21), a(22)}$, and the corresponding quantization error is
\[V_4=\frac 12 r^2(V_2+V_2)=r^2V_2=0.00352544,\]
which is the lemma.
\end{proof}

\begin{prop} \label{prop2}
Let $n\geq 2$, and $\ga_n$ be an optimal set of $n$-means. Then, $\ga_n$ does not contain any point from the open interval $(r, 1-r)$, i.e., $\ga_n\ii (r, 1-r)=\es$.
\end{prop}

\begin{proof} By Lemma~\ref{lemma11}, Lemma~\ref{lemma12}, and Lemma~\ref{lemma15}, the proposition is true for $n=2, 3, 4$. We now prove that the proposition is true for $n=5$. Let $\ga_5:=\set{a_1, a_2, a_3, a_4, a_5}$ be an optimal set of five-means, such that $0<a_1<a_2<a_3<a_4<a_5<1$.
Consider the set of five points $\gb$ given by $\gb:=\set{a(11), (12), a(21), a(221), a(222)}$. The distortion error due to the set $\gb$ is given by
\begin{align*}
& \int\min_{b\in \gb} (x-b)^2 dP=3 \int_{J_{11}}(x-a(11))^2 dP+2\int_{J_{221}}(x-a(221))^2 dP=0.00281089.
\end{align*}
Since $V_5$ is the quantization error for five-means, we have $V_5\leq 0.00281089$.
If $0.189261=S_{11}(1)<a_1$, then
\[V_5\geq \int_{J_{11}}(x-S_{11}(1))^2 dP=0.00312009>V_5,\]
which gives a contradiction. Hence, we can assume that $a_1<S_{11}(1)=0.189261$. Similarly, $S_{22}(0)<a_5$.
For the sake of contradiction, assume that $\ga_5$ contains a point from $(r, 1-r)$. Notice that due to Proposition~\ref{prop0}, if $\ga_5$ contains a point from $(r, 1-r)$, then it cannot contain more than one point from $(r, 1-r)$. Suppose that $a_2 \in (r, 1-r)$. Two cases can arise:

Case~1: $a_2 \in [\frac 12, 1-r)$.

Then,
$a_1\leq 0.189261=S_{11}(1)<S_{121}(0)<S_{121}(1)=0.328117<0.344631=\frac 12(0.189261 +\frac{1}{2})<0.352705=S_{122}(0)$, and so,
\begin{align*}
V_5&\geq \int_{J_{11}}(x-a(11))^2 dP+\int_{J_{121}}(x-S_{11}(1))^2 dP+\int_{J_{122}}(x-\frac 12)^2 dP=0.00364889>V_5,
\end{align*}
which is a contradiction.

Case~2: $a_2\in (r, \frac 12]$.

Then, $S_{1211} (1)=0.2816<0.281742=\frac 12(a(11, 1211)+r)<0.292297=S_{1212}(0)$ implying the fact that $J_{11}\uu J_{1211}\sci M(a(11, 1211)|\ga_5)$ and $J_{122}\sci M(r|\ga_5)$. Again, we have
\begin{align*} \label{eq90}
&\int_{J_2}\min_{a\in \set{a_2, a_3, a_4, a_5}}(x-a)^2 dP\geq \int_{J_2}\min_{a\in S_2(\ga_4)}(x-a)^2 dP=\frac 12  \int_{J_2}\min_{a\in S_2(\ga_4)}(x-a)^2 d(P\circ S_2^{-1})\\
&=\frac 12  \int \min_{a\in S_2(\ga_4)}(S_2(x)-a)^2 dP=\frac 12  \int \min_{a\in \ga_4}(S_2(x)-S_2(a))^2 dP=\frac 12 r^6 V,
\end{align*}
where $\ga_4$ is an optimal set of four-means.
Thus, we obtain
\begin{align*}  V_5 \geq \int_{J_{11}\uu J_{1211}}(x-a(11, 1211))^2 dP+\int_{J_{122}}(x-r)^2 dP+\frac 12 r^6 V=0.00294718>V_5
\end{align*}
which gives a contradiction.

Hence, we can assume that $a_2 \not \in (r, 1-r)$. Likewise, if $a_3 \in (r, 1-r)$, we can show that a contradiction arises. Proceeding in the similar fashion, one can show that the proposition is true for all $6\leq n\leq 15$. We now give the general proof of the proposition for all $n\geq 16$. Let $\ga_n$ be an optimal set of $n$-means for all $n\geq 16$, and $V_n$ is the corresponding quantization error.
Consider the set of sixteen points $\gb$ given by $\gb:=\set{a(\gs) : \gs \in I^4}$. The distortion error due to the set $\gb$ is given by
\begin{align*}
& \int\min_{b\in \gb} (x-b)^2 dP=r^8V=0.00012628.
\end{align*}
Since $V_n$ is the quantization error for $n$-means for $n\geq 16$, we have
$V_{n}\leq V_{16}\leq 0.00012628$.
Write $\ga_n:=\set{a_1, a_2, \cdots, a_n}$, where $0<a_1<a_2<\cdots <a_n<1$. By Lemma~\ref{lemma13}, we see that $\ga_n\ii J_1\neq \es$ and $\ga_n\ii J_2\neq \es$. Let $j$ be the largest positive integer such that $a_j\in J_1$. Then, $a_{j+1}>r$. We need to show that $\ga_n\ii (r, 1-r)=\es$. For the sake of contradiction, assume that $\ga_n\ii (r, 1-r)\neq \es$. Proposition~\ref{prop0} implies that if $\ga_n$ contains a point from the open interval $(r, 1-r)$, then it can not contain more than one point from the open interval $(r, 1-r)$. Thus, we have $a_j\leq r<a_{j+1}<1-r\leq a_{j+2}$.
The following two cases can arise:

Case 1: $a_{j+1}\in [\frac 12, 1-r)$.

Then, by Proposition~\ref{prop0}, we have $\frac 12(a_j+a_{j+1})<r$ implying $a_j<2r-a_{j+1}\leq 2r-\frac 12=0.370082<S_{12 2 1 2}(0)=0.372942$. Thus,
\[V_n\geq \int_{J_{12 2 1 2}\uu J_{1222}}(x- 0.3700816)^2 dP=0.000150535>V_{16}\geq V_n,\]
which is a contradiction.

Case 2: $a_{j+1} \in (r, \frac 12]$.

Since this case is the reflection of Case~1 with respect to the point $\frac 12$, a contradiction arises.

Hence, $\ga_n$ does not contain any point from the open interval $(r, 1-r)$. Thus, the proof of the proposition is complete.
\end{proof}

\begin{prop}\label{prop3}

Let $\ga_n$ be an optimal set of $n$-means with $n\geq 2$. Then, the Voronoi region of any point in $\ga_n\ii J_1$ does not contain any point from $J_2$, and the Voronoi region of any point in $\ga_n\ii J_2$ does not contain any point from $J_1$.

\end{prop}

\begin{proof}  Notice that $\frac 12(a(1)+a(2))=\frac 12$, $r<\frac 12 (a(1)+a(21))<1-r$, $r<\frac 12 (a(12)+a(2))<1-r$, and $r<\frac 12=\frac 12 (a(12)+a(21))<1-r$. Thus, by Lemma~\ref{lemma11}, Lemma~\ref{lemma12}, and Lemma~\ref{lemma15}, the proposition is true for $n=2, 3, 4$. It can also be shown that the proposition is true for $5\leq n\leq 7$. Due to lengthy as well as the technicality of the proofs we don't show them in the paper, and give a general proof of the proposition for all $n\geq 8$.
Let us consider a set of eight points $\gb$ given by $\gb:=\set{a(\gs) : \gs \in \set{1, 2}^3}$. Then,
\[\int\min_{a \in \gb} (x-a)^2 dP=\sum_{\gs \in \set{1, 2}^3} \mathop{\int}\limits_{J_{\gs}}(x-a(\gs))^2 dP=0.000667229.\]
Since $V_n$ is the $n$th quantization error for $n\geq 8$, we have $V_n\leq V_{8}\leq 0.000667229$.
Let $\ga_n:=\set{a_1, a_2, \cdots, a_n}$ be an optimal set of $n$-means for $n\geq 8$ with $0\leq a_1<a_2<\cdots <a_n\leq 1$, and let $j$ be the greatest positive integer such that $a_j\in J_1$. Then, by Proposition~\ref{prop2}, we have $a_j< r$ and $1-r<a_{j+1}$. Suppose that the Voronoi region of $a_{j+1}$ contains points from $J_1$. Then, $\frac 12(a_j+a_{j+1})<r$ yielding $a_j< 2r-a_{j+1}\leq 2r-(1-r)=3r-1=0.305124<0.312534=S_{12122}(0)$. Hence, by Corollary~\ref{cor1},
\begin{align*}
V_n&\geq \int_{J_{12122}\uu J_{122}}(x- 0.305124)^2 dP=0.00107592>V_{8}\geq V_n,
\end{align*}
which is a contradiction. Thus, we can assume that the Voronoi region of $a_{j+1}$ does not contain any point from $J_1$. Similarly, we can show that the Voronoi region of $a_{j}$ does not contain any point from $J_2$. Hence, the proposition is true for all $n\geq 8$. Thus, we complete the proof of the proposition.
\end{proof}

We are now ready to give the proof of Theorem~\ref{Th1}.

\subsection*{Proof of Theorem~\ref{Th1}.} We prove the theorem by induction. For $n\geq 2$ let $\ga_n$ be an optimal set of $n$-means for $P$. By Lemma~\ref{lemma11}, Lemma~\ref{lemma12}, and Lemma~\ref{lemma15}, the theorem is true for $n=2, 3, 4$. Suppose that the assertion of the theorem holds for all $m<n$, where $n\geq 2$. Set $\ga_1:=\ga_n\ii J_1$ and $\ga_2:=\ga_n\ii J_2$, and $j:=\te{card}(\ga_1)$.
By Lemma~\ref{lemma13}, Lemma~\ref{lemma14},  Proposition~\ref{prop2}, and Proposition~\ref{prop3}, there exists a $j\in \set{1, 2, \cdots, n-1}$ such that
\[V_n=\frac 12 r^2(V_j+V_{n-j}),\]
which is same as the expression of $V_n$ given in \cite{GL3} for $r=\frac 13$.   Without any loss of generality, we can assume that $j\geq n-j$. Then, proceeding similarly, as given in the proof of Theorem~5.2 in \cite{GL3}, we can show that the following inequalities are true:
 \[2^{\ell(n)-1}\leq j\leq 2^{\ell(n)} \te{ and } 2^{\ell(n)-1}\leq n-j< 2^{\ell(n)}.\]
The rest of the induction hypothesis, follows exactly same as the last part of the proof of Theorem~5.2 in \cite{GL3}. Thus, the proof of Theorem~\ref{Th1} is complete. \qed

\section{Proof of the main theorem Theorem~\ref{maintheorem}}

In this section, we determine the least upper bound of $r$ for which $\gb_n(I)$ forms an optimal set of $n$-means. It is known that $\gb_n(I)$ forms a CVT if $0<r \leq\frac{5-\sqrt{17}}2$ (see \cite[Lemma~4.2]{R2}). By Proposition~\ref{prop51}, $\gg_n(I)$ forms a CVT if $0.3613249509\leq r\leq 0.4376259168<\frac{5-\sqrt{17}}2$.
Let us now prove the following proposition.

\begin{prop} \label{prop4} For $n\geq 2$, let $\gb_n(I)$ be the set defined by Definition~\ref{def1}, and $\gg_n(I)$ be the set defined by Definition~\ref{def2}. Assume that $n$ is not of the form $2^{\ell(n)}$ for any positive integer $\ell(n)$. Then, $V(\gb_n(I))>V(\gg_n(I))$ if $0.4350411707< r\leq 0.4376259168<\frac{5-\sqrt{17}}2$, and $V(\gb_n(I))=V(\gg_n(I))$ if $r=0.4350411707$, where $V(\gb_n(I)):=V(P, \gb_n(I))$ and $V(\gg_n(I)):=V(P, \gg_n(I))$, respectively, denote the distortion errors for the CVTs $\gb_n(I)$ and $\gg_n(I)$.
\end{prop}

\begin{proof} If $n$ is of the form $2^{\ell(n)}$ for some positive integer $\ell(n)$, then as $\gb_n(I)=\gg_n(I)$, we have $V(\gb_n(I))=V(\gg_n(I))$ for all $0<r<\frac 12$. Let us assume that $n$ is not of the form $2^{\ell(n)}$ for any positive integer $\ell(n)\geq 2 $. Then, the following three cases can aries:

Case~1. $n=3$.

In this case we have, $V(\gb_3(I))=-\frac{(r-1) \left(r^4+r^2\right)}{8 (r+1)}$ and $V(\gg_3(I))=-\frac{r^7+r^6+4 r^5-2 r^4-2 r^3-8 r^2+9 r-3}{48 (r+1)}$. Then, $V(\gb_3(I))=V(\gg_3(I))$ if $r=0.4350411707$, and $V(\gb_3(I))>V(\gg_3(I))$ if $0.4350411707<r<\frac 12$.

Case~2. $n\geq 4$ and $2^{\ell(n)}< n\leq 3 \cdot 2^{\ell(n)-1}$.

Then, using Definition~\ref{def2} and Proposition~\ref{prop34}, we see that $V(\gb_n(I))=V(\gg_n(I))$ if

$\frac{1}{ 2^{\ell(n)}} r^{2\ell(n)} V \Big(2^{\ell(n)+1}-n+r^2(n-2^{\ell(n)})\Big)=\frac 1{2^{\ell(n)-1}}r^{2(\ell(n)-1)} \Big(V(\gg_3(I))\, (n-2^{\ell(n)}) +V(\gg_2(I)) \,(3\cdot 2^{\ell(n)-1}-n)\Big)$, which after simplification yields that $\frac{1}{2} r^2 (r^2+1) V=V(\gg_3(I))$, i.e., $V(\gb_3(I))=V(\gg_3(I))$. Hence, by Case~1, we have $V(\gb_n(I))=V(\gg_n(I))$ if $r=0.4350411707$, and $V(\gb_n(I))>V(\gg_n(I))$ if $0.4350411707<r<\frac 12$.

Case~3. $n\geq 4$ and $ 3 \cdot 2^{\ell(n)-1}<n< 2^{\ell(n)+1}$.

Then, using Definition~\ref{def2} and Proposition~\ref{prop34}, we see that $V(\gb_n(I))=V(\gg_n(I))$ if

$\frac{1}{ 2^{\ell(n)}} r^{2\ell(n)} V \Big(2^{\ell(n)+1}-n+r^2(n-2^{\ell(n)})\Big)=\frac 1{2^{\ell(n)-1}}r^{2(\ell(n)-1)} \Big(V(\gg_3(I))  (2^{\ell(n)+1}-n) +V(\gg_4(I)) (n-3\cdot 2^{\ell(n)-1})\Big) $, which after simplification yields that $\frac{1}{2} r^2 (r^2+1) V=V(\gg_3(I))$, i.e., $V(\gb_3(I))=V(\gg_3(I))$. Hence, by Case~1, we have $V(\gb_n(I))=V(\gg_n(I))$ if $r=0.4350411707$, and $V(\gb_n(I))>V(\gg_n(I))$ if $0.4350411707<r<\frac 12$.

Recall that both $\gb_n(I)$ and $\gg_n(I)$ form CVTs if  $0.3613249509\leq r\leq 0.4376259168<\frac{5-\sqrt{17}}2$. Hence, by Case~1, Case~2, and Case~3, we see that if $n$ is not of the form $2^{\ell(n)}$ for any positive integer $\ell(n)$. Then, $V(\gb_n(I))>V(\gg_n(I))$ if $0.4350411707< r\leq 0.4376259168<\frac{5-\sqrt{17}}2$, and $V(\gb_n(I))=V(\gg_n(I))$ if $r=0.4350411707$. Thus, the proof of the proposition is complete.
\end{proof}

We now give the following definition.

\begin{defi}\label{def3} For $n\in \D N$ with $n\geq 2$ let $\ell(n)$ be the unique natural number with $2^{\ell(n)} \leq n<2^{\ell(n)+1}$. Let $\gd_n:=\gd_n(I)$ be the set defined as follows: $\gd_2:=\set{a(1), a(2)}$, and $\gd_3:=\set{a(11, 121, 1221), a(1222, 21), a(22)}$ or $\gd_3:=\set{a(11),  a(12, 2111), a(2112, 212, 22)}$. For $n\geq 4$, define $\gd_n:=\gd_n(I)$ as follows:
\[\gd_n(I)=\left\{\begin{array}{cc}
\mathop{\uu}\limits_{\go\in I} S_\go(\gd_3)\uu \mathop{\uu}\limits_{\go \in \set{1, 2}^{\ell(n)-1}\setminus I} S_\go (\gd_2) & \te{ if } 2^{\ell(n)}\leq n\leq 3\cdot 2^{\ell(n)-1},\\
\mathop{\uu}\limits_{\go\in \set{1, 2}^{\ell(n)-1}\setminus I} S_\go(\gd_3) \uu \mathop{\uu}\limits_{\go \in I} S_\go(\gd_4) & \te{ if } 3\cdot 2^{\ell(n)-1}<n< 2^{\ell(n)+1},
\end{array}
\right.\] where $I \sci \set{1, 2}^{\ell(n)-1}$ with $\te{card}(I)=n-2^{\ell(n)}$ if $2^{\ell(n)}\leq n\leq 3\cdot 2^{\ell(n)-1}$; and
$\te{card}(I)=n-3 \cdot 2^{\ell(n)-1}$ if $3\cdot 2^{\ell(n)-1}<n< 2^{\ell(n)+1}$.
\end{defi}

The following proposition is due to \cite{R2}.

\begin{prop}(see \cite[Proposition~4.3]{R2}) \label{prop5}
Let $\gd_n(I)$ be the set defined by Definition~\ref{def3}, and $\gb_n(I)$ be the set defined by Definition~\ref{def1}. Suppose that $n$ is not of the form $2^{\ell(n)}$ for any positive integer $\ell(n)$. Then
$V(P, \gd_n(I))<V(P, \gb_n(I))$ if $0.4371985206<r\leq \frac{5-\sqrt{17}}2$, where $V(P, \gd_n(I))$ and $V(P, \gb_n(I))$, respectively, denote the distortion errors for the CVTs $\gd_n(I)$ and $\gb_n(I)$.
\end{prop}

We are now ready to give the proof of the main theorem Theorem~\ref{maintheorem}.

\subsection* {Proof of Theorem~\ref{maintheorem}.} Recall that $\gb_n(I)$ forms a CVT if $0<r\leq\frac{5-\sqrt{17}}2$. In \cite{GL3}, it is shown that $\gb_n(I)$ forms an optimal set of $n$-means if $r=\frac 13<0.4350411707$. Theorem~\ref{Th1} implies that $\gb_n(I)$ also forms an optimal set of $n$-means if $r= 0.4350411707$. Proposition~\ref{prop4} implies that
 $V(P, \gb_n(I))=V(P, \gg_n(I))$ if $r=0.4350411707$, and $V(P, \gb_n(I))>V(P, \gg_n(I))$  if $0.4350411707< r\leq 0.4376259168<\frac{5-\sqrt{17}}2$. By Proposition~\ref{prop5}, it follows that if $0.4371985206<r\leq \frac{5-\sqrt{17}}2$, then $V(P, \gb_n(I))>V(P, \gd_n(I))$. Hence, if $0.4350411707<r\leq\frac{5-\sqrt{17}}2$, then the set $\gb_n(I)$ forms a CVT but does not form an optimal set of $n$-means.
Thus, the least upper bound of $r$ for which $\gb_n(I)$ forms an optimal set of $n$-means is given by $r=0.4350411707$, and this completes the proof of the theorem. \qed

\subsection*{Acknowledgement} The author is grateful to the referees for their valuable comments and suggestions.

\end{document}